\documentclass{llncs}
\usepackage{amsmath} 
\usepackage{amssymb}
\usepackage{epsfig}
\usepackage{blkarray}
\usepackage{blkarray}
\newcommand{\DSD}{{\rm DSD}}
\newcommand{\diag}{{\rm diag}}

\newcommand{\vol}{{\rm vol}}

\title{Computing Diffusion State Distance using Green's Function and
  Heat Kernel on Graphs}
\author{
Edward Boehnlein\inst{1}
\and
Peter Chin \inst{2} 
\thanks{This author is supported in part by NSF grant DMS 1222567 as well as AFOSR grant FA9550-12-1-0136.}
\and
Amit Sinha \inst{3} 
\and
Linyuan Lu \inst{4} 
\thanks{Research supported in part by NSF
grant DMS 1300547 and ONR grant N00014-13-1-0717.}
}

\institute{University of South Carolina, Columbia, SC 29208, USA,
\email{boehnlei@email.sc.edu}
\and 
Boston University, Boston, MA 02215, USA,
\email{spchin@cs.bu.edu}
\and 
Boston University, Boston, MA 02215, USA,
\email{amits@bu.edu}
\and
University of South Carolina, Columbia, SC 29208, USA,
\email{lu@math.sc.edu}
}

\begin{document}
\maketitle

\begin{abstract}
  The diffusion state distance (DSD) was introduced by
  Cao-Zhang-Park-Daniels-Crovella-Cowen-Hescott [{\em PLoS ONE, 2013}]
  to capture functional similarity in protein-protein interaction
  networks. They proved the convergence of DSD for non-bipartite graphs.
  In this paper, we extend the DSD to bipartite graphs using
  lazy-random walks and consider the general $L_q$-version of DSD.
  We discovered the connection between the DSD $L_q$-distance and 
  Green's function, which was studied  by Chung and Yau [{\em
J. Combinatorial Theory (A), 2000}]. Based on that, we computed 
the DSD $L_q$-distance for Paths, Cycles, Hypercubes, as well as
random graphs $G(n,p)$ and $G(w_1,\ldots, w_n)$. We also examined
the DSD distances of two biological networks. 
\end{abstract}

\section{Introduction}
Recently, the diffusion state distance (DSD, for short) was introduced in \cite{DSD}
to capture functional similarity in protein-protein interaction (PPI)
networks. The diffusion state distance  is much more effective than the classical 
shortest-path distance for the problem of transferring functional
labels across nodes in PPI networks,  based on evidence presented in \cite{DSD}.
The definition of DSD is purely graph theoretic and is based on random walks.

Let $G=(V,E)$ be a simple undirected graph on the vertex set 
$\{v_1,v_2,\ldots, v_n\}$. For any two vertices $u$ and $v$,
let $He^{\{k\}}(u,v)$ be the expected number of times that a random walk
starting at node $u$ and proceeding for $k$ steps, will visit node
$v$.  Let $He^{\{k\}}(u)$ be the vector
$(He^{\{k\}}(u,v_1),\ldots, He^{\{k\}}(u,v_n))$.
The diffusion state distance (or DSD, for short)  between two vertices $u$ and
$v$ is defined as 
$$DSD(u,v)=\lim_{k\to\infty}\left\|He^{\{k\}}(u)-He^{\{k\}}(v)\right\|_1$$
provided the limit exists (see \cite{DSD}).  Here the $L_1$-norm is
not essential. Generally, for $q\geq 1$, 
one can define the DSD $L_q$-distance as
$$DSD_q(u,v)=\lim_{k\to\infty}\left\|He^{\{k\}}(u)-He^{\{k\}}(v)\right\|_q$$
provided the limit exists. (We use $L_q$ rather than $L_p$ to avoid confusion, as $p$
will be used as a probability throughout the paper.)

In \cite{DSD}, Cowen et al. showed that the above limit always exists
whenever the random walk on $G$ is ergodic (i.e., $G$ is connected
non-bipartite graph). They also prove that this distance can be
computed by the following formula:
$$DSD(u,v)=\|(1_u-1_v)(I-D^{-1}A+W)^{-1}\|_1$$
where $D$ is the diagonal degree matrix, $A$ is the adjacency matrix,
and $W$ is the constant matrix in which each row is a
copy of $\pi$ , $\pi=\frac{1}{\sum_{i=1}^nd_i}(d_1, \ldots, d_n)$ is the unique steady state distribution.

A natural question is how to define the diffusion state distance for a
bipartite graph. We suggest to use the lazy random walk.
For a given $\alpha \in (0,1)$, one can choose to stay at the current
node $u$ with probability $\alpha$, and choose to move to one of its
neighbors with probability $(1-\alpha)/d_u$. In other words, the
transitive matrix of the $\alpha$-lazy random walk is
$$T_\alpha=\alpha I + (1-\alpha)D^{-1}A.$$

Similarly, 
let $He^{\{k\}}_\alpha(u,v)$ be the expected number of times that the
$\alpha$-lazy random walk
starting at node $u$ and proceeding for $k$ steps, will visit node
$v$.  Let $He^{\{k\}}_\alpha(u)$ be the vector
$(He^{\{k\}}_\alpha(u,v_1),\ldots, He^{\{k\}}_\alpha(u,v_n))$.
The $\alpha$-diffusion state distance $L_q$-distance between two vertices $u$ and $v$ is 
$$DSD^\alpha_q(u,v)=\lim_{k\to\infty}\left\|He^{\{k\}}_\alpha(u)-He^{\{k\}}_\alpha(v)\right\|_q.$$

\begin{theorem}\label{t1}
  For any connected graph $G$ and $\alpha\in (0,1)$, the 
$DSD^\alpha_q(u,v)$ is always
  well-defined and satisfies
  \begin{equation}
    \label{eq:4}
DSD^\alpha_q(u,v)=(1-\alpha)^{-1}\| ({\bf 1}_u-{\bf 1}_v) {\mathbb G}\|_q.     
  \end{equation}
Here $\mathbb G$ is the matrix of  Green's function of $G$.
\end{theorem}

Observe that $(1-\alpha)DSD^\alpha_q(u,v)$ is independent of the choice of
$\alpha$. Naturally, we define
the DSD $L_q$-distance of any graph $G$ as:
$$DSD_q(u,v):=\lim_{\alpha\to 0} (1-\alpha)DSD^\alpha_q(u,v)
=\| ({\bf 1}_u-{\bf 1}_v) {\mathbb G}\|_q.$$
This definition extends the original definition for non-bipartite graphs.

With properly chosen $\alpha$,
$\|He^{\{k\}}_\alpha(u)-He^{\{k\}}_\alpha(v)\|_q$ converges faster
than \linebreak $\|He^{\{k\}}(u)-He^{\{k\}}(v)\|_q$. This fact leads to a faster algorithm 
to estimate a single distance $DSD_q(u,v)$ using random walks. We will discuss it in
Remark \ref{r1}. 

Green's function was introduced in 1828 by George Green \cite{Green}
to solve some partial differential equations, and it has found many applications
(e.g. \cite{BGT}, \cite{CY},\cite{book},  \cite{Duffy}, \cite{Hedin}, \cite{SH}).

The Green's function on graphs  was first 
investigated by Chung and Yau \cite{CY} in 2000.
Given a graph $G=(V,E)$ and a given function $g\colon V\to {\mathbb R}$,
consider the problem to find $f$ satisfying
the discrete Laplace equation
$$L f= \sum_{y\in V} (f(x)-f(y))p_{xy}=g(x).$$
Here $p_{xy}$ is the transition probability of the random walk from
$x$ to $y$. Roughly speaking, Green's function is the left inverse
operator of $L$ (for the graphs with boundary). 
It is closely related to the Heat kernel of the
graphs (see also \cite{Davies}) and the normalized Laplacian.

% In the graph context, we follow the notation of \cite{CY}: take $S \subset V$ and 
% $\delta S := \{ y\notin S \colon \exists x \in S \text{ with } p_{xy} \ne 0\} \ne \emptyset$. Define $\mathcal{L}_S$ to
% be the submatrix of $\mathcal{L}$ corresponding only to those rows and columns indexed by vertices in $S$. Then 
% for $t \ge 0$, the heat kernal of $S$ is $\mathcal{H}_t = e^{-t \mathcal{L}_S}$. 

In this paper, we will use Green's function to compute the DSD
$L_q$-distance for various graphs. The maximum DSD $L_q$-distance
varies from graphs to graphs. The maximum 
DSD $L_q$-distance for paths and cycles are at the order of
$\Theta(n^{1+1/q})$ while the $L_q$-distance for some random graphs
$G(n,p)$ and $G(w_1,\ldots, w_n)$ are constant for some ranges of
$p$. The hypercubes are somehow between the two classes. The DSD
$L_1$-distance is $\Omega(n)$ while the $L_q$-distance is $\Theta(1)$
for $q>1$. Our method for random graphs is based on the strong
concentration of the Laplacian eigenvalues.

The paper is organized as follows. In Section 2, we will briefly
review the terminology on the Laplacian eigenvalues, Green's Function,
and heat kernel. The proof of Theorem \ref{t1} will be proved in
Section 3. In Section 4, we apply Green's function to calculate the DSD
distance for various symmetric graphs like paths, cycles, and hypercubes.
We will calculate the DSD $L_2$-distance for random graphs $G(n,p)$
and $G(w_1,w_2,\ldots, w_n)$ in Section 5. In the last section, we
examined two brain networks: a cat and a Rhesus monkey. The
distributions of the DSD distances are calculated.

%The DSD distances of some real
%information networks will be presented in the last section.

\section{Notation and background}
In this paper, we only consider undirected simple graph $G=(V,E)$ with
the vertex set $V$ and the edge set $E$. For each vertex $x\in V$, the
{\em neighborhood} of $x$, denoted by $N(x)$, is the set of vertices
adjacent to $x$. The {\em degree} of $x$, denoted by $d_x$, is the
cardinality of $N(x)$. We also denote the maximum degree by $\Delta$
and the minimum degree by $\delta$.

Without loss of generalization, we assume that the set of vertices
is ordered and assume $V=[n]=\{1,2,\ldots, n\}$.
Let $A$ be the adjacency matrix and
$D=\diag(d_1,\ldots,d_{n})$ be the diagonal
matrix of degrees.  For a given subset $S$, let the volume of $S$
to be $\vol(S):=\sum_{i\in S}d_i$. In particular, we write
$\vol(G)=\vol(V)=\sum_{i=1}^n d_i$.

Let $V^*$ be the linear space of all real functions on $V$.
The {\em discrete Laplace operator} $L \colon V^* \to V^*$ is defined
as
$$L(f)(x)=\sum_{y\in N(x)}\frac{1}{d_x}(f(x)-f(y)).$$
The Laplace operator can also written as a $(n\times n)$-matrix:
$$L=I-D^{-1}A.$$
Here $D^{-1}A$ is the transition probability matrix of the (uniform) random walk on
$G$. Note that $L$ is not symmetric. We consider a symmetric version
$${\cal L}:=I-D^{-1/2}AD^{-1/2}=D^{1/2}LD^{-1/2},$$
which is so called the {\em normalized Laplacian}. Both $L$ and ${\cal L}$
have the same set of eigenvalues.
The eigenvalues of ${\cal L}$ can be listed as
$$0=\lambda_0\leq \lambda_1\leq \lambda_2\leq \cdots \leq
\lambda_{n-1}\leq 2.$$
The eigenvalue $\lambda_1>0$ if and only if $G$ is connected while
$\lambda_{n-1}=2$ if and only if $G$ is a bipartite graph. Let $\phi_0,
\phi_1,\ldots, \phi_{n-1}$ be a set of orthogonal unit eigenvectors. 
Here $\phi_0=\frac{1}{\sqrt{\vol(G)}}(\sqrt{d_1},\ldots, \sqrt{d_n})$ is the positive unit
eigenvector for $\lambda_0=0$ and $\phi_i$ is the eigenvector for
$\lambda_i$ ($1\leq i\leq n-1$).

Let $O=(\phi_0,\ldots,\phi_{n-1})$ and $\Lambda=\diag(0, \lambda_1,
\ldots, \lambda_{n-1})$. Then $O$ is an orthogonal matrix and
${\cal L}$ be diagonalized as
\begin{equation}
  \label{eq:31}
{\cal L}=O\Lambda O'.  
\end{equation}

Equivalently, we have
\begin{equation}
  \label{eq:3}
L=D^{-1/2}O\Lambda O'D^{1/2}.  
\end{equation}

The {\em Green's function} ${\mathbb G}$ is the matrix with its entries, indexed
by vertices $x$ and $y$, defined by a set of two equations:
\begin{align}
\label{eq:g1}
{\mathbb G}L(x,y)&=I(x,y)-\frac{d_y}{\vol(G)}, \\
\label{eq:g2}
 {\mathbb G}{\bf 1}&=0.  
\end{align}
(This is the so-called Green's function for graphs without boundary in \cite{CY}.)

The {\em normalized Green's function} ${\cal G}$ is defined similarly:
$${\cal GL}(x,y)=I(x,y)-\frac{\sqrt{d_xd_y}}{\vol(G)}.$$
The matrices $\mathbb{G}$ and $\mathcal{G}$ are related by
$${\cal G}=D^{1/2}{\mathbb G}D^{-1/2}.$$
Alternatively, $\cal G$ can be defined using the eigenvalues and
eigenvectors of $\cal L$ as follows:
$${\cal G}=O\Lambda^{\{-1\}}O',$$
where $\Lambda^{\{-1\}}=\diag(0, \lambda_1^{-1},
\ldots, \lambda_{n-1}^{-1})$.
Thus, we have
\begin{equation}
  \label{eq:5}
  {\mathbb G}(x,y)=\sum_{l=1}^{n-1}\frac{1}{\lambda_l} \sqrt{\frac{d_y}{d_x}}\phi_l(x)\phi_l(y).
\end{equation}
 
For any real $t\geq 0$, the heat kernel ${\cal H}_t$ is defined as
$${\cal H}_t=e^{-t{\cal L}}.$$
Thus, $${\cal H}_t(x,y)=\sum_{l=0}^{n-1} e^{-\lambda_i
  t}\phi_l(x)\phi_l(y).$$
The heat kernel ${\cal H}_t$ satisfies the heat equation
$$\frac{d}{dt}{\cal H}_tf=-{\cal L}{\cal H}_tf.$$
The relation of the heat kernel and Green's function is  given by
$${\cal G}=\int_0^\infty {\cal H}_tdt -\phi_0'\phi_0.$$
The heat kernel can be used to compute Green's function for the
Cartesian product of two graphs. We will omit the details here.
Readers are directed to  \cite{CY} and \cite{fan4} for the further
information.

\section{Proof of main theorem}
\begin{proof}[Proof of Theorem \ref{t1}:]
%The proof is the same as in \cite{DSD} for the usual random walk ($\alpha=0$).
%For completeness of our argument we reproduce the argument here. 
Rewrite the transition probability matrix $T_\alpha$ as
\begin{align*}
T_\alpha&=\alpha I +(1-\alpha)D^{-1}A.\\
&=D^{-1/2}(\alpha I + (1-\alpha) D^{-1/2}AD^{-1/2})D^{1/2}\\
&=D^{-1/2}(\alpha I+(1-\alpha)(I-{\cal L}))D^{1/2}\\
&=D^{-1/2}(I-(1-\alpha){\cal L})D^{1/2}.
\end{align*}
For $k=0,1,\ldots, n-1$, let $\lambda^*_k=1-(1-\alpha)\lambda_k$
and $\Lambda^*=diag(\lambda^*_0,\ldots, \lambda^*_{n-1})=I -(1-\alpha)\Lambda$.
Applying Equation \eqref{eq:3}, we get
$$T_\alpha=D^{-1/2}O\Lambda^* O'D^{1/2}
=(O'D^{1/2})^{-1}\Lambda^* O'D^{1/2}.$$ 
Then for any $t \ge 1$, the $t$-step transition matrix is 
$T_\alpha^{t } =  (OD^{1/2})^{-1} \Lambda^{*t} OD^{1/2}=D^{-1/2}O\Lambda^{*t} O'D^{1/2}
.$ 
Denote $p^{\{t\}}_\alpha (u,j)$ as the $(u,j)^{th}$ entry in $T_\alpha^{t}$.
\begin{align*}
p^{\{t\}}_\alpha(u,j)&= \sum_{l=0}^{n-1}  (\lambda_l^*)^t \sqrt{\frac{d_j}{d_u}} 
\phi_l(u)\phi_l(j)\\
&= \frac{d_j}{\vol(G)} +\sum_{l=1}^{n-1}(\lambda_l^*)^t \sqrt{\frac{d_j}{d_u}} \phi_l(u)\phi_l(j).
\end{align*}
Thus,
$$He^{\{k\}}_\alpha(u,j)-
He^{\{k\}}_\alpha(v,j)=\sum_{t=0}^k\sum_{l=1}^{n-1} (\lambda_l^*)^t
d_j^{1/2}\phi_l(j)( d_u^{-1/2}\phi_l(u)-d_v^{-1/2}\phi_l(v)).$$
The limit $\lim_{k\to\infty}He^{\{k\}}_\alpha(u,j)-
He^{\{k\}}_\alpha(v,j)$ forms the sum of $n$ geometric series:
$$ \sum_{t=0}^\infty \sum_{l=1}^{n-1} (\lambda_l^*)^t
d_j^{1/2}\phi_l(j)( d_u^{-1/2}\phi_l(u)-d_v^{-1/2}\phi_l(v)).$$
Note each geometric series converges since the common ratio
$\lambda_l^*\in (-1,1)$. Thus,

\begin{align*}
\lim_{k\to\infty}\left(He^{\{k\}}_\alpha(u,j)-
He^{\{k\}}_\alpha(v,j)\right) &=  \sum_{t=0}^\infty \sum_{l=1}^{n-1} (\lambda_l^*)^t
d_j^{1/2}\phi_l(j)( d_u^{-1/2}\phi_l(u)-d_v^{-1/2}\phi_l(v))\\
&= \sum_{l=1}^{n-1} d_j^{1/2}\phi_l(j)(
d_u^{-1/2}\phi_l(u)-d_v^{-1/2}\phi_l(v))
\sum_{t=0}^\infty(\lambda_l^*)^t\\
&= \sum_{l=1}^{n-1} \frac{1}{1-\lambda_l^*}d_j^{1/2}\phi_l(j)(
d_u^{-1/2}\phi_l(u)-d_v^{-1/2}\phi_l(v))\\
&=\frac{1}{1-\alpha}\sum_{l=1}^{n-1} \frac{1}{\lambda_l}d_j^{1/2}\phi_l(j)(
d_u^{-1/2}\phi_l(u)-d_v^{-1/2}\phi_l(v))\\
&=\frac{1}{1-\alpha} ({\mathbb G}(u,j)-\mathbb{G}(v,j)).
\end{align*}
We have
$$\lim_{k\to\infty}
He_\alpha^{\{k\}}(u)-He_\alpha^{\{k\}}(v)=\frac{1}{1-\alpha}
({\bf 1}_u-{\bf 1}_v) \mathbb{G}.$$

\end{proof}

\begin{remark}\label{r1}
Observe that the convergence rate of $He_\alpha^{\{k\}}(u)-He_\alpha^{\{k\}}(v)$ 
is determined by 
 $\bar \lambda^*:=\max\{1-(1-\alpha)\lambda_1,
(1-\alpha)\lambda_{n-1}-1)$. 
It is critical that we assume $\alpha\ne 0$. 
When $\alpha=0$ then $\bar \lambda^*<1$ holds only if
$\lambda_{n-1}<2$, i.e. $G$ is a non-bipartite graph (see \cite{DSD}).

When $\lambda_1+\lambda_{n-1}>2$, 
$\bar \lambda^*$ (as a function of $\alpha$) achieves 
the minimum value
$\frac{\lambda_{n-1}-\lambda_1}{\lambda_{n-1}+\lambda_1}$
at $\alpha=1-\frac{2}{\lambda_1+\lambda_{n-1}}$.
This is the best mixing rate that the $\alpha$-lazy random walk on $G$
can achieve. Using the $\alpha$-lazy random walks (with
$\alpha=1-\frac{2}{\lambda_1+\lambda_{n-1}}$) to approximate the DSD
$L_q$-distance will be faster than using regular random walks.
  \end{remark}
Equation \eqref{eq:5} implies $\|\mathbb{G}\|_2\leq
\frac{1}{\lambda_1}\sqrt{\frac{\Delta}{\delta}}$. Combining with Theorem
\ref{t1}, we have
  \begin{corollary}
    For any connected simple graph $G$, and any two vertices $u$ and
    $v$, we have
$DSD_2(u,v)\leq \frac{\sqrt{2}}{\lambda_1}\sqrt{\frac{\Delta}{\delta}}$.
  \end{corollary}
%   \begin{proof}
%   Applying Theorem \ref{t1}, we have
%   \begin{align*}
%     DSD_2(u,v)&= \|({\bf 1}_u-{\bf 1}_v)\mathbb{G}\|_2\\
% &=\|({\bf 1}_u-{\bf 1}_v)D^{-1/2} \Lambda^{\{-1\}}D^{1/2}\|_2\\
% &\leq \|({\bf 1}_u-{\bf 1}_v)\|_2 \|D^{-1/2}\|_2
% \|\Lambda^{\{-1\}}\|_2 \|D^{1/2}\|_2\\
% &\leq \sqrt{2}\delta^{-1/2}\frac{1}{\lambda_1}\Delta^{1/2}\\
% &= \frac{\sqrt{2}}{\lambda_1}\sqrt{\frac{\Delta}{\delta}}.
%   \end{align*}
%   \end{proof}

Note that for any connected graph $G$ with diameter $m$ (Lemma 1.9, \cite{fan4}) 
$$\lambda_1>\frac{1}{m\;\vol(G)}.$$
This implies a uniform bound for the DSD $L_2$ distances on any
connected graph $G$ on $n$ vertices.
$$DSD_2(u,v)\leq \sqrt{\frac{2\Delta}{\delta}}m\;
\vol(G)<\sqrt{2}n^{3.5}.$$
This is a very coarse upper bound. But it does raise an interesting
question ``How large can the DSD $L_q$-distance be?''

\section{Some examples of the DSD distance}
In this section, we use Green's function to compute the DSD
$L_q$-distance (between two vertices of the distance reaching the diameter)
 for
paths,
cycles, and hypercubes.
\subsection{The path $P_n$}
We label the vertices of $P_n$ as $1,2, \dots, n$, in sequential order. 
Chung and Yau computed the Green's function $\cal G$ of the weighed path
with no boundary (Theorem 9, \cite{CY}). 
It implies that Green's
function of the path $P_n$ is given by:
for any $u\leq v$,
\begin{align*}
\mathcal{G}(u,v) &= \frac{\sqrt{d_ud_v}}{4(n-1)^2}  \bigg(\sum_{z<u} (d_1+\ldots+d_z)^2
+ \sum_{v \le z} (d_{z+1}+ \dots +  d_n)^2 \\
&\hspace*{3mm}
- \sum_{u \le z <v} (d_1 + \dots + d_z)(d_{z+1}+ \dots +  d_n)    \bigg) \\
&=  \frac{\sqrt{d_ud_v}}{4(n-1)^2} \left( \sum_{z=1}^{u-1} (2z-1)^2  
+\sum_{z=v}^{n-1} (2n-2z-1)^2 - \sum_{z=u}^{v-1} (2z-1)(2n-2z-1)
\right)\\
&=  \frac{\sqrt{d_ud_v}}{4(n-1)^2} \bigg(\sum_{z=1}^{n-1} (2z-1)^2  
+\sum_{z=v}^{n-1} (2n-2)(2n-4z) 
- \sum_{z=u}^{v-1} (2z-1)(2n-2)\bigg)\\
&=  \frac{\sqrt{d_ud_v}(2n-1)(2n-3)}{12(n-1)} 
+\frac{\sqrt{d_ud_v}}{2(n-1)} \bigg( \sum_{z=v}^{n-1}(2n-4z) 
- \sum_{z=u}^{v-1} (2z-1)\bigg)\\
&= \frac{\sqrt{d_ud_v}}{2(n-1)}\left((u-1)^2+(n-v)^2-\frac{2n^2-4n+3}{6}\right).
\end{align*}
When $u>v$, we have
$$\mathcal{G}(u,v)={\mathcal G}(v,u)=\frac{\sqrt{d_ud_v}}{2(n-1)}\left((v-1)^2+(n-u)^2-\frac{2n^2-4n+3}{6}\right).$$
Applying
$\mathbb{G}(u,v)=\frac{\sqrt{d_v}}{\sqrt{d_u}}\mathcal{G}(u,v)$,
we get
$$\mathbb{G}(u,v)=
\begin{cases}
  \frac{d_v}{2(n-1)}\left((u-1)^2+(n-v)^2-\frac{2n^2-4n+3}{6}\right)
& \mbox{ if } u\leq v;\\
  \frac{d_v}{2(n-1)}\left((v-1)^2+(n-u)^2-\frac{2n^2-4n+3}{6}\right)
& \mbox{ if } u> v.
\end{cases}
$$
We have
\begin{align*}
  \mathbb{G}(1,1)&=\frac{4n^2-8n+3}{12(n-1)};\\
 \mathbb{G}(1,j)&=\frac{1}{n-1}\left((n-j)^2-\frac{2n^2-4n+3}{6}\right) \quad
\mbox{ for } 2\leq j\leq n-1;\\
  \mathbb{G}(1,n)&=-\frac{2n^2-4n+3}{12(n-1)};\\
  \mathbb{G}(n,1)&=-\frac{2n^2-4n+3}{12(n-1)};\\
 \mathbb{G}(n,j)&=\frac{1}{n-1}\left((j-1)^2-\frac{2n^2-4n+3}{6}\right) \quad
\mbox{ for } 2\leq j\leq n-1;\\
  \mathbb{G}(n,n)&=\frac{4n^2-8n+3}{12(n-1)}.
\end{align*}

Thus, 
\begin{equation}
  \mathbb{G}(1,j)-\mathbb{G}(n,j)=
  \begin{cases}
\frac{n-1}{2} & \mbox{ if }j=1;\\
n+1-2j & \mbox{ if }2\leq j\leq n-1;\\
-\frac{n-1}{2} & \mbox { if }j=n.
  \end{cases}
\end{equation}

\begin{theorem}
For any $q\geq 1$, the DSD $L_q$-distance of the Path $P_n$ between $1$ and $n$
satisfies
$$DSD_q(1,n)=(1+q)^{-1/q} n^{1+1/q} +O(n^{1/q}).$$
\end{theorem}
\begin{proof}
  \begin{align*}
    DSD_q(1,n)&= \left(
      2\left(\frac{n-1}{2}\right)^q+\sum_{j=2}^{n-1}|n+1-2j|^q\right)^{1/q}\\
&= \left(\frac{1}{1+q}n^{1+q}+O(n^q)\right)^{1/q}\\
&=(1+q)^{-1/q} n^{1+1/q} +O(n^{1/q}).
  \end{align*}
\end{proof}

For $q=1$, we have the following exact result:
\begin{align*}
DSD_1(1,n)&=\sum_{j=1}^{n}|\mathbb{G}(1,j)-\mathbb{G}(n,j)|\\
&=
\begin{cases}
  2k^2-2k+1 & \mbox{ if }n=2k\\
  2k^2 &\mbox{ if }n=2k+1.
\end{cases}
\end{align*}

% \begin{align*}
% DSD_2(1,n)&=\sqrt{\sum_{j=1}^{n}(\mathbb{G}(1,j)-\mathbb{G}(n,j))^2}\\
% &=
% \begin{cases}
%   \sqrt{\frac{8}{3}k^3-6k^2+\frac{16}{3}k-\frac{3}{2}} & \mbox{ if }n=2k\\
%    \sqrt{\frac{8}{3}k^3-2k^2+\frac{4}{3}k} &\mbox{ if }n=2k+1.
% \end{cases}
% \end{align*}

\subsection{The cycle $C_n$}
% Let $\omega$ denote the $n^{th}$ root of unity. Then the adjacency matrix $A$ for $C_n$ has eigenvalues
% $2 \cos{\frac{j 2 \pi}{n}}$ for $j = 0, \dots , n-1$, and eigenvectors 
% $\alpha_j = (1, \omega^j , \dots , \omega^{j(n-1)})^T$. The eigenvalues of the Laplacian $\mathcal{L}$ 
% are given by $1-\cos{\frac{j 2 \pi}{n}}$. Then, in the $\ell_2$ norm, we have
% \begin{eqnarray*}
% DSD(u,v) &=& \| (1_u - 1_v) \mathcal{G}(u,v) \|_2 \\
% &=& \|  (1_u -1_v) \sum_{j=1}^{n-1} \frac{1}{1-\cos{ \frac{2j \pi}{n}}}  \omega^{j(u+v)} \|_2 \\
% &=& \sqrt{\sum_{k=1}^n 
% \left(  \sum_{j=1}^{n-1} \frac{1}{1-\cos{ \frac{2j \pi}{n}}} (\omega^{juk}-\omega^{jvk})\right)^2}
% \end{eqnarray*}

Now we consider Green's function of cycle $C_n$.
For $x,y\in \{1,2,\ldots, n\}$, let $|x-y|_c$ be the graph distance of
$x,y$ in $C_n$. We have the following Lemma.

\begin{lemma}\label{lcn}
For even $n=2k$,  Green's function ${\mathbb G}$ of $C_n$ is given by
 $${\mathbb
  G}(x,y)=\frac{1}{2k}(k-|x-y|_c)^2-\frac{k}{6}-\frac{1}{12k}.$$
 For odd $n=2k+1$,   Green's function ${\mathbb G}$ of $C_n$ is given by
$${\mathbb G}(x,y)=\frac{2}{2k+1}{k+1-|x-y|_c\choose 2}-\frac{k^2+k}{3(2k+1)}.$$
\end{lemma}
\begin{proof}
 We only prove the even case here. The odd case is similar and will
 be left to the readers.

For $n=2k$, it suffices to verify that
${\mathbb G}$ satisfies Equations \eqref{eq:g1} and \eqref{eq:g2}.
To verify Equation \eqref{eq:g1}, we need show
$$ {\mathbb G}(x,y)-\frac{1}{2}{\mathbb G}(x,y-1)-\frac{1}{2}{\mathbb G}(x,y+1)=
\begin{cases}
-\frac{1}{n} \mbox{ if } x\ne y;\\
1-\frac{1}{n} \mbox{ if } x= y.\\  
\end{cases}
$$
Let $z=\frac{k}{6}+\frac{1}{12k}$ and $i=|x-y|_c$.
For $x\ne y$,  we have
\begin{align*}
&\hspace*{-1cm}   {\mathbb G}(x,y)-\frac{1}{2}{\mathbb G}(x,y-1)-\frac{1}{2}{\mathbb
    G}(x,y+1)\\
&=(\frac{1}{2k}(k-i)^2-z) -\frac{1}{2}(\frac{1}{2k}(k-i-1)^2-z)
-\frac{1}{2}(\frac{1}{2k}(k-i+1)^2-z)\\
&=-\frac{1}{2k}\\
&=-\frac{1}{n}.
\end{align*}
When $x=y$, we have 
\begin{align*}
&\hspace*{-1cm}  {\mathbb G}(x,y)-\frac{1}{2}{\mathbb G}(x,y-1)-\frac{1}{2}{\mathbb
    G}(x,y+1)\\
&=\frac{1}{2k}k^2-z -\frac{1}{2}\left(\frac{1}{2k}(k-1)^2-z\right)
-\frac{1}{2}\left(\frac{1}{2k}(k-1)^2-z\right)\\
&=\frac{2k-1}{2k}\\
&=1-\frac{1}{n}.
\end{align*}
To verify Equation \eqref{eq:g2}, it is enough to verify
$$1^2+2^2+\cdots + (k-1)^2+k^2+(k-1)^2+\cdots+1^2=\frac{2k^3+k}{3}=n^2z.$$
This can be done by induction on $k$.
 \end{proof}

\begin{theorem}
For any $q\geq 1$, the DSD $L_q$-distance of the Cycle $C_n$ between
$1$ and $\lfloor \frac{n}{2}\rfloor +1$
satisfies
$$DSD_q(1,\lfloor \frac{n}{2}\rfloor +1)=\left(\frac{4}{1+q}\right)^{1/q} \left(\frac{n}{4}\right)^{1+1/q} +O(n^{1/q}).$$
\end{theorem}
\begin{proof}
We only verify the case of even cycle here. The odd cycle is similar
and will be omitted.

For $n=2k$, the difference of ${\mathbb G}(1,j)$ and ${\mathbb
  G}(1+k,j)$ have a simple form:
$${\mathbb G}(1,j)-{\mathbb
  G}(1+k,j)=\frac{1}{2k}((k-i)^2-i^2)=\frac{k}{2}-i,$$
where $i=|j-1|_c$. Thus,
\begin{align*}
    DSD_q(1,1+k)&= \left(2\sum_{i=0}^{k-1}\left|\frac{k}{2}-i\right|^q\right)^{1/q}\\
&= \left(\frac{4}{1+q}\left(\frac{k}{2}\right)^{1+q}+O(k^q)\right)^{1/q}\\
&=\left(\frac{4}{1+q}\right)^{1/q} \left(\frac{n}{4}\right)^{1+1/q} +O(n^{1/q}).
  \end{align*}
\end{proof}

\subsection{The hypercube $Q_n$}
Now we consider the hypercube $Q_n$, whose vertices are the binary
strings of length $n$ and whose edges are pairs of vertices differing
only at one coordinate. Chung and Yau \cite{CY} computed the Green's
function of $Q_n$: for any two vertices $x$ and $y$ with distance $k$
in $Q_n$, 
\begin{align*}
{\mathbb G}(x,y)\!&=\!2^{-2n} \!\! \left(\!-\!\sum_{ j < k} 
\frac{({n \choose 0} + \dots + {n \choose j} )
({n \choose j+1} + \dots + {n \choose n})}{{n-1\choose j}}
\!+\!\sum_{k \leq j}
 \frac{({n \choose j+1} + \dots + {n \choose n})^2}{{n-1 \choose
     j}}\!\!\right)\\
&=2^{-2n}\sum_{j=0}^n
 \frac{({n \choose j+1} + \dots + {n \choose n})^2}{{n-1 \choose
     j}} -2^{-n} \sum_{ j < k} 
\frac{{n \choose j+1} + \dots + {n \choose n}}{{n-1\choose j}}.
\end{align*}

We are interested in the DSD distance between a pair of antipodal
vertices. Let $\bf 0$ denote the all-$0$-string and  $\bf 1$ denote
the all-$1$-string. For any vertex $x$, if the distance between $\bf
0$ and $x$ is $i$ then the distance between $\bf
1$ and $x$ is $n-i$.
We have
\begin{align}
\nonumber
{\mathbb G}({\bf 0}, x)-{\mathbb G}({\bf 1}, x) &=
-2^{-n} \sum_{ j < k} 
\frac{{n \choose j+1} + \dots + {n \choose n}}{{n-1\choose j}}
+ 2^{-n} \sum_{ j < n-k} 
\frac{{n \choose j+1} + \dots + {n \choose n}}{{n-1\choose j}}\\
&=2^{-n}\sum_{j=k}^{n-k-1} 
\frac{{n \choose j+1} + \dots + {n \choose n}}{{n-1\choose j}}.
\label{eq:8}
\end{align}
Here we use the convention that $\sum_{j=b}^ac_j=-\sum_{j=a}^bc_j$ for $b>a$.
\begin{theorem}
For any $q\geq 1$, the DSD $L_q$-distance of the hypercube $Q_n$ between
${\bf 0}$ and $\bf 1$
satisfies
\begin{equation}
  \label{eq:9}
  DSD_q({\bf 0},{\bf 1})=\left(\sum_{k=0}^{n} {n\choose k} 
\left|2^{-n}\sum_{j=k}^{n-k-1} 
\frac{{n \choose j+1} + \dots + {n \choose n}}{{n-1\choose j}}\right|^q
\right)^{1/q}.
\end{equation}
In particular, $DSD_q({\bf 0},{\bf 1})=\Theta(1)$ when $q>1$ while $DSD_1({\bf 0},{\bf 1})=\Omega(n)$.
\end{theorem}
\begin{proof}
Equation \eqref{eq:9} follows from the definition of DSD
$L_q$-distance
and Equation \eqref{eq:8}. Let 
$$a_k={n\choose k} 
\left| 2^{-n}\sum_{j=k}^{n-k-1} 
\frac{{n \choose j+1} + \dots + {n \choose n}}{{n-1\choose
    j}}\right|^q.$$
Observe that $a_k=a_{n-k}$, we only need to estimate $a_k$ for $0\leq
k\leq n/2$. 
Also we can throw away the terms in the second summation for $j>n/2$
since that part is at most half of $a_k$.
For $k\leq j\leq  n/2$, 
$$\frac{1}{2}\leq 2^{-n}\left({n \choose j+1} + \dots + {n \choose n}\right)\leq 1.$$
Thus $a_k$ has the same magnitude as
$b_k:={n\choose k} 
\left(\sum_{j=k}^{n/2} 
\frac{1}{{n-1\choose
    j}}\right)^q.$

For $q>1$, we first bound $b_k$ by $b_k\leq  {n\choose k} 
\left(
\frac{n/2}{{n-1\choose
    k}}\right)^q=O(n^{(1-q)k+q}).$
When $k>\frac{q+2}{q-1}$, we have $b_k=O(n^{-2})$. The total
contribution of those $b_k$'s is $O(n^{-1})$, which is negligible. 
Now consider the term $b_k$ for $k=0,1,\ldots, \lfloor \frac{q+2}{q-1}
\rfloor$. We bound $b_k$ by
$$b_k\leq {n\choose k}\left(\frac{1}{{n-1\choose k}}+  \frac{n/2}{{n-1\choose
    k+1}}\right)^q=O(1).$$
This implies $DSD_q({\bf 0},{\bf 1})=O(1)$. The lower bound 
$DSD_q({\bf 0},{\bf 1}) \ge 1$ is obtained by taking the term at $k=0$.
Putting together, we have $DSD_q({\bf 0},{\bf 1})=\Theta(1)$ for
$q>1$.

For $q=1$, note that 
$$b_k=\sum_{j=k}^{n/2}\frac{{n\choose k}}{{n-1
    \choose j}} >\frac{{n\choose k}}{{n-1
    \choose k}}=\frac{n}{n-k}>1.$$
Thus, $DSD_1({\bf 0},{\bf 1})=\Omega(n)$.

\end{proof}

\section{Random graphs}
In this section, we will calculate  the 
DSD $L_q$-distance in two random graphs models. For random graphs, 
the non-zero Laplacian eigenvalues of a
graph $G$ are often concentrated around $1$.
The following Lemma is useful to the DSD $L_q$-distance.

\begin{lemma} \label{l1}
  Let $\lambda_1,\ldots, \lambda_{n-1}$ be all non-zero Laplacian
  eigenvalues of a graph $G$. Suppose there is a small number  $\epsilon \in
  (0,1/2)$, so that for $1\leq i\leq n-1$, $|1-\lambda_i|\leq
  \epsilon$.
Then for any pairs of vertices $u, v$, the DSD $L_q$-distance
satisfies
\begin{align}
  \label{eq:upperb1}
  |DSD_q(u,v)-2^{1/q}|\leq \frac{\epsilon}{1-\epsilon}
\sqrt{\frac{\Delta}{d_u}+\frac{\Delta}{d_v}} 
\quad \mbox{ if } q\geq 2,\\
\label{eq:upperb2}
 |DSD_q(u,v)-2^{1/q}|\leq n^{\frac{1}{q}-\frac{1}{2}}\frac{\epsilon}{1-\epsilon}
\sqrt{\frac{\Delta}{d_u}+\frac{\Delta}{d_v}} 
\quad \mbox{ for } 1\leq q < 2.
\end{align}
\end{lemma}
\begin{proof}
  Rewrite the normalized Green's function ${\cal G}$ as
$${\cal G}=I-\phi_0'\phi_0+\Upsilon.$$
Note that the eigenvalues of $\Upsilon:={\cal G}-I+\phi_0\phi_0'$
are $0, \frac{1}{\lambda_1}-1, \ldots,  \frac{1}{\lambda_{n-1}}-1$.
Observe that for each $i=1,2,\ldots, n-1$, 
$|\frac{1}{\lambda_{i}}-1|\leq \frac{\epsilon}{1-\epsilon}$.
We have 
 $$\|\Upsilon\|\leq \frac{\epsilon}{1-\epsilon}.$$
Thus,
\begin{align*}
  \DSD_q(u,v) &=\|({\bf 1}_u-{\bf 1}_v)D^{-1/2}{\cal G}D^{1/2}\|_q\\
&=\|({\bf 1}_u-{\bf 1}_v)D^{-1/2}(I-\phi_0'\phi+\Upsilon)D^{1/2}\|_q\\
&\leq \|({\bf 1}_u-{\bf 1}_v)D^{-1/2}(I-\phi_0'\phi)D^{1/2}\|_q
+\|({\bf 1}_u-{\bf 1}_v)D^{-1/2}\Upsilon D^{1/2}\|_q.
\end{align*}
Viewing $\Upsilon$ as the error term, we first calculate the main
term.
\begin{align*}
&\hspace*{-1cm} \|({\bf 1}_u-{\bf 1}_v)D^{-1/2}(I-\phi_0'\phi)D^{1/2}\|_q\\
 &=\|({\bf 1}_u-{\bf 1}_v)(I- W)\|_q\\
&=\|({\bf 1}_u-{\bf 1}_v)\|_q\\
&=2^{1/q}.
\end{align*}
The $L_2$-norm of the error term  can be bounded by
\begin{align*}
 &\hspace*{-1cm} \|({\bf 1}_u-{\bf 1}_v)D^{-1/2}\Upsilon D^{1/2}\|_2\\
 &\leq \|({\bf 1}_u-{\bf 1}_v)D^{-1/2}\|_2 \|\Upsilon\| \|D^{1/2}\|\\
 &\leq \sqrt{\frac{1}{d_u}+\frac{1}{d_v}} \frac{\epsilon}{1-\epsilon}
 \sqrt{\Delta}\\
&=\frac{\epsilon}{1-\epsilon}
\sqrt{\frac{\Delta}{d_u}+\frac{\Delta}{d_v}}.
\end{align*}
To get the bound of $L_q$-norm from $L_2$-norm,
 we apply the following relation of $L_q$-norm and $L_2$-norm
to the error term.
For any vector $x\in {\mathbb R}^n$,
$$\|x\|_q\leq \|x\|_2 \quad \mbox{ for } q\geq 2.$$
and
$$\|x\|_q\leq n^{\frac{1}{q}-\frac{1}{2}}\|x\|_2 \quad \mbox{ for } 1\leq q<2.$$

The inequalities \eqref{eq:upperb1} and \eqref{eq:upperb2}
follow from  the triangular inequality of the $L_q$-norm and
the upper bound of the error term.
\end{proof}

Now we consider the classical Erd\H{o}s-Renyi random graphs $G(n,p)$.
For a given $n$ and $p\in (0,1)$,  $G(n,p)$ is a
random graph on the vertex set $\{1,2,\ldots, n\}$ obtained by
adding
each pair $(i,j)$ to the edges of $G(n,p)$ with probability $p$
independently.

There are plenty of references on the concentration of the
eigenvalues of $G(n,p)$ (for example,  \cite{CR}, \cite{OL},\cite{KS}, and \cite{grgraph}).
Here we list some facts on $G(n,p)$.
\begin{enumerate}
\item For $p>\frac{(1+\epsilon)\log n}{n}$, almost surely 
$G(n,p)$ is connected.

\item For $p\gg \frac{\log n}{n}$, $G(n,p)$ is ``almost regular'';
  namely for all vertex $v$, $d_v=(1+o_n(1))np$.

\item For $np(1-p)\gg \log^4 n$, all non-zero  Laplacian eigenvalues
  $\lambda_i$'s satisfy (see \cite{grgraph})
  \begin{equation}
    \label{eq:gnp}
|\lambda_i-1|\leq \frac{(3+o_n(1))}{\sqrt{np}}.    
  \end{equation}
\end{enumerate}

Apply Lemma \ref{l1} with $\epsilon =
\frac{(3+o_n(1))}{\sqrt{np}}$, and note that $G(n,p)$ is
almost-regular. We get the following theorem.
\begin{theorem}
  For $p(1-p)\gg \frac{\log^4 n}{n}$, almost surely for
all pairs of vertices $(u,v)$, the DSD $L_q$-distance of $G(n,p)$ satisfies
$$DSD_q(u,v)=2^{1/q} \pm  O \left(\frac{1}{\sqrt{np}}\right) \quad \mbox{ if }
q\geq 2,$$
$$DSD_q(u,v)=2^{1/q} \pm O \left(\frac{n^{\frac{1}{q}-\frac{1}{2}}}{\sqrt{np}}\right) \quad \mbox{ if }
1\leq q<2.$$
\end{theorem}

Now we consider the random graphs with given expected degree sequence
$G(w_1,\ldots, w_n)$ (see \cite{BHL}, \cite{CL02a}, \cite{CL02b},
\cite{book}, 
\cite{Hofstad2013}).
It is defined
as follows:
\begin{enumerate}
\item Each vertex $i$ (for $1\leq i \leq n$) is associated with a
  given positive weight $w_i$.
\item Let $\rho=\frac{1}{\sum_{i=1}^nw_i}$.
For each pair of vertices $(i,j)$, 
$ij$ is added as an edge with probability $w_iw_j\rho$ independently.
($i$ and $j$ may be equal so loops are allowed. Assume $w_iw_j\rho\leq
1$ for $i,j$.)
\end{enumerate}

Let $w_{min}$  be the minimum weight.
There are many references on the concentration of the
eigenvalues of $G(w_1,\ldots, w_n)$ (see \cite{CLV1}, \cite{CLV2}, \cite{CR},
 \cite{OL}, \cite{grgraph}).
The version used here is in \cite{grgraph}.
%Here are some facts about $G(w_1,\ldots, w_n)$. 
\begin{enumerate}
\item For each vertex $i$, the expected degree of $i$ is $w_i$.
\item Almost surely for all $i$ with $w_i\gg \log n$, then the degree
  $d_i=(1+o(1))w_i$.
\item If $w_{min} \gg \log^4 n$, all non-zero  Laplacian eigenvalues
  $\lambda_i$ (for $1\leq i\leq n-1$), 
  \begin{equation}
    \label{eq:gw}
    |1-\lambda_i|\leq \frac{3+o_n(1)}{\sqrt{w_{min}}}.
  \end{equation}
\end{enumerate}

\begin{theorem}
 Suppose $w_{min} \gg \log^4 n$, almost surely for
all pairs of vertices $(u,v)$, the DSD $L_q$-distance of
$G(w_1,\ldots, w_n)$
 satisfies
$$DSD_q(u,v)=2^{1/q} \pm O\left(\frac{1}{\sqrt{w_{min}}}
\sqrt{\frac{w_{max}}{w_u}+\frac{w_{max}}{w_v}}
\right)\quad \mbox{ if } q\geq 2,
$$
$$DSD_q(u,v)=2^{1/q} \pm O\left(\frac{n^{\frac{1}{q}-\frac{1}{2}}}{\sqrt{w_{min}}}
\sqrt{\frac{w_{max}}{w_u}+\frac{w_{max}}{w_v}}
\right)\quad \mbox{ if } 1\leq q < 2.
$$
\end{theorem}
 
\section{Examples of  biological  networks}
In this section, we will examine the distribution of the DSD distances for some
biological networks.
 The set of graphs analyzed in this section
include three graphs of brain data from the Open Connectome Project
\cite{OCP} and two more graphs built from the \textit{S. cerevisiae}
PPI network and \textit{S. pombe} PPI network used in
\cite{DSD}. Figure 1 and 2 serves as a visual representation of one of the
two brain data graphs: the graph of a cat and  the graph of a Rhesus monkey.
The network of the cat brain has 65 nodes and 1139 edges while
the network of rhesus monkey brain has 242 nodes and 4090 edges.

\begin{figure}[htbp]
 \centering
%\includegraphics[
%width=0.45\linewidth,height=0.40\linewidth]{cat_cluster.png}
\psfig{figure=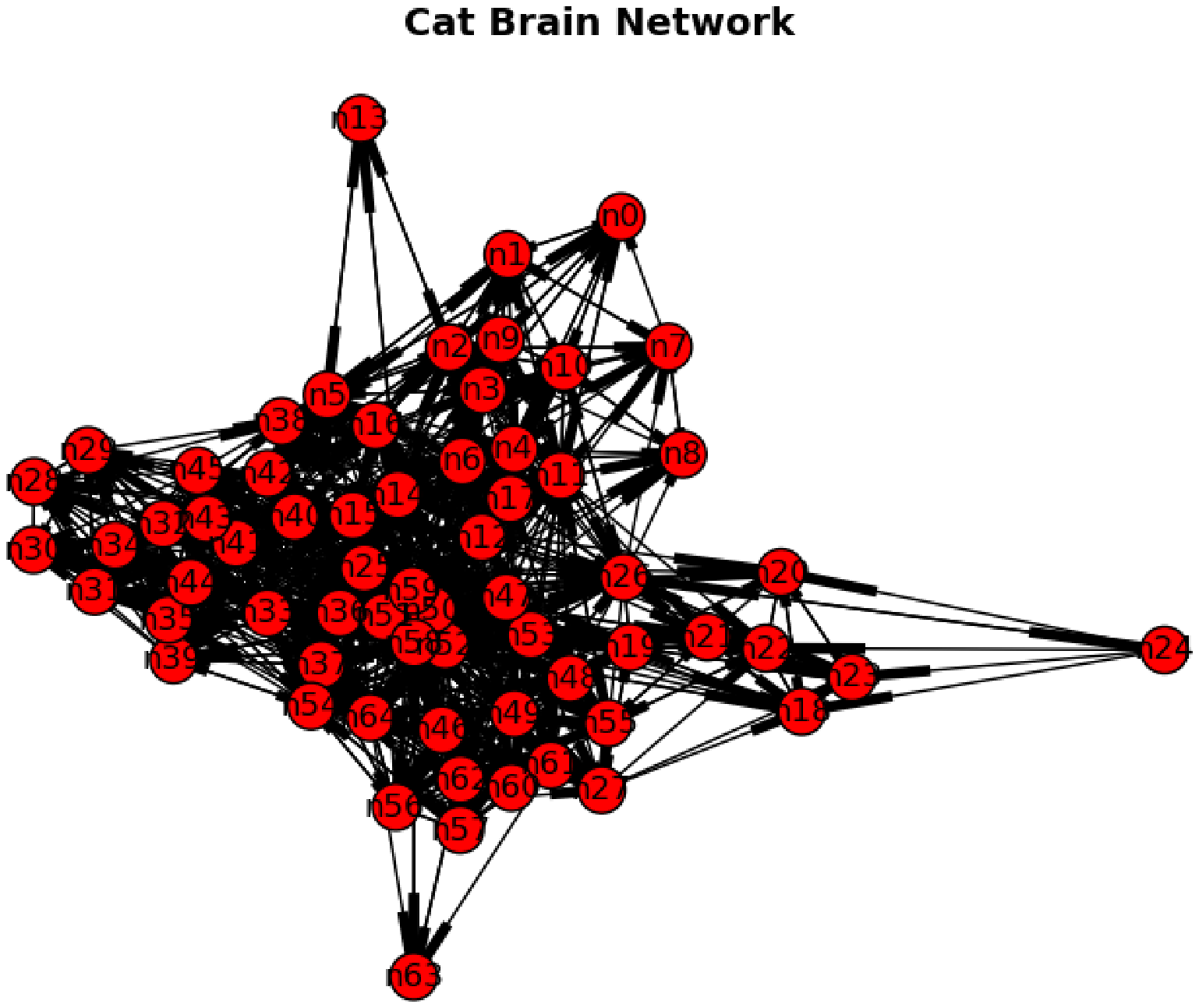, width=0.45\linewidth, height=0.40\linewidth}
\hfil
%\includegraphics[
%width=0.45\linewidth,height=0.40\linewidth]{rhesus_cluster.png}
\psfig{figure=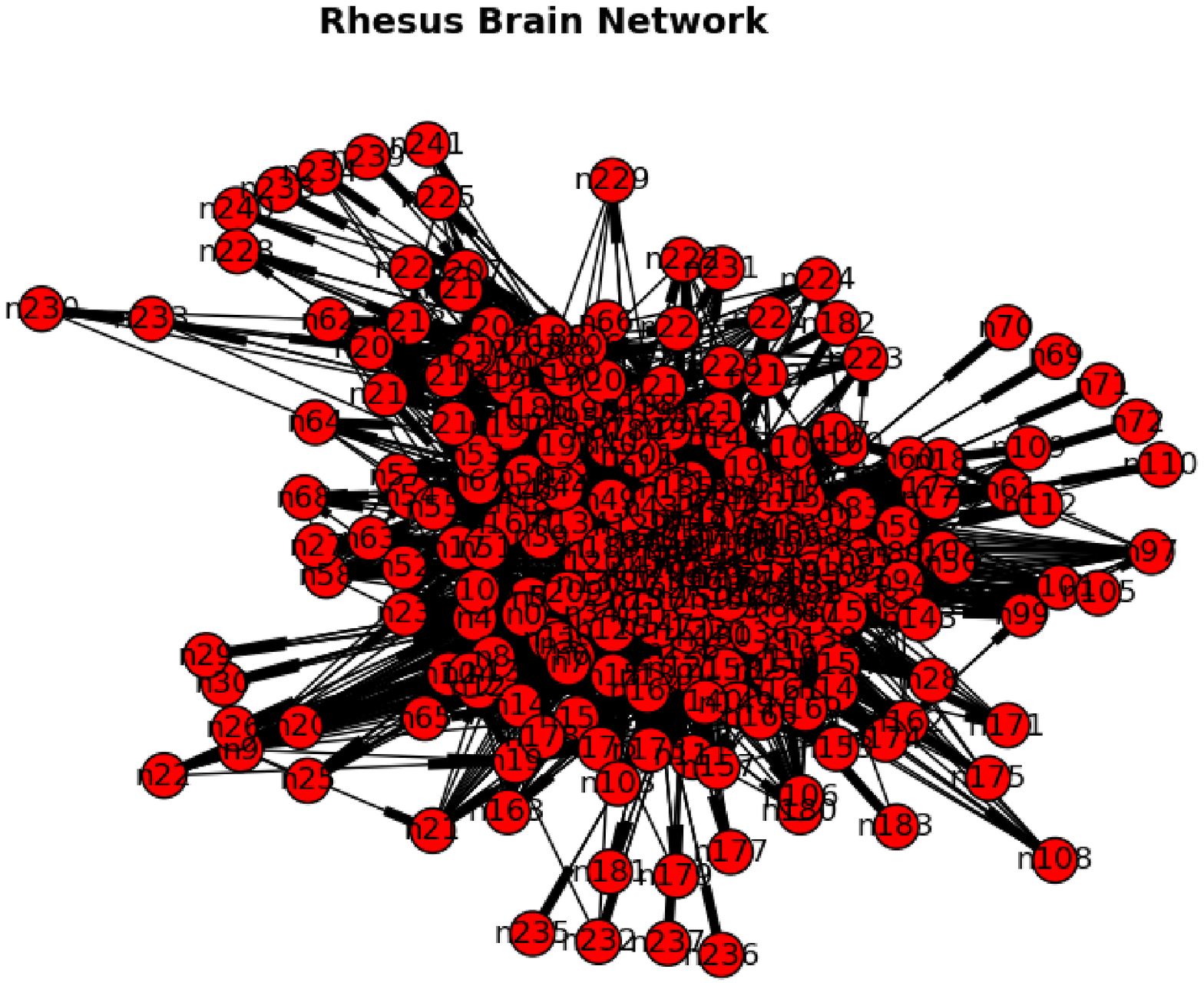, width=0.45\linewidth, height=0.40\linewidth}
% rhesus_cluster.png: 812x612 pixel, 100dpi, 20.62x15.54 cm, bb=0 0 585 441
\caption{The brain networks: (a), a Cat; (b): a Rhesus Monkey} 
\label{fig:1}
\end{figure}

 Each node in the
Rhesus graph represents a region in the cerebral cortex originally
analyzed in \cite{Olaf}. Each edge represents axonal connectivity
between regions and there is no distinction between strong and weak
connections in this graph \cite{Olaf}.  The Cat data-set follows a
similar pattern where each node represents a region of the brain and
each edge represents connections between them. The Cat data-set
represents 18 visual regions, 10 auditory regions, 18 somatomotor
regions, and 19 frontolimbic regions\cite{Reus}. 

For each network above, we calculated all-pair DSD $L_1$-distances. 
Divide the possible values into many small intervals 
and compute the number of pairs falling into each interval. The
results are shown in Figure \ref{fig:1}. The patterns are quite
surprising to us.

\begin{figure}[htbp]
 \centering
%\includegraphics[%bb=0 0 585 441,
%width=0.45\linewidth,height=0.45\linewidth]{cat_DSD_no_space.png}
\psfig{figure=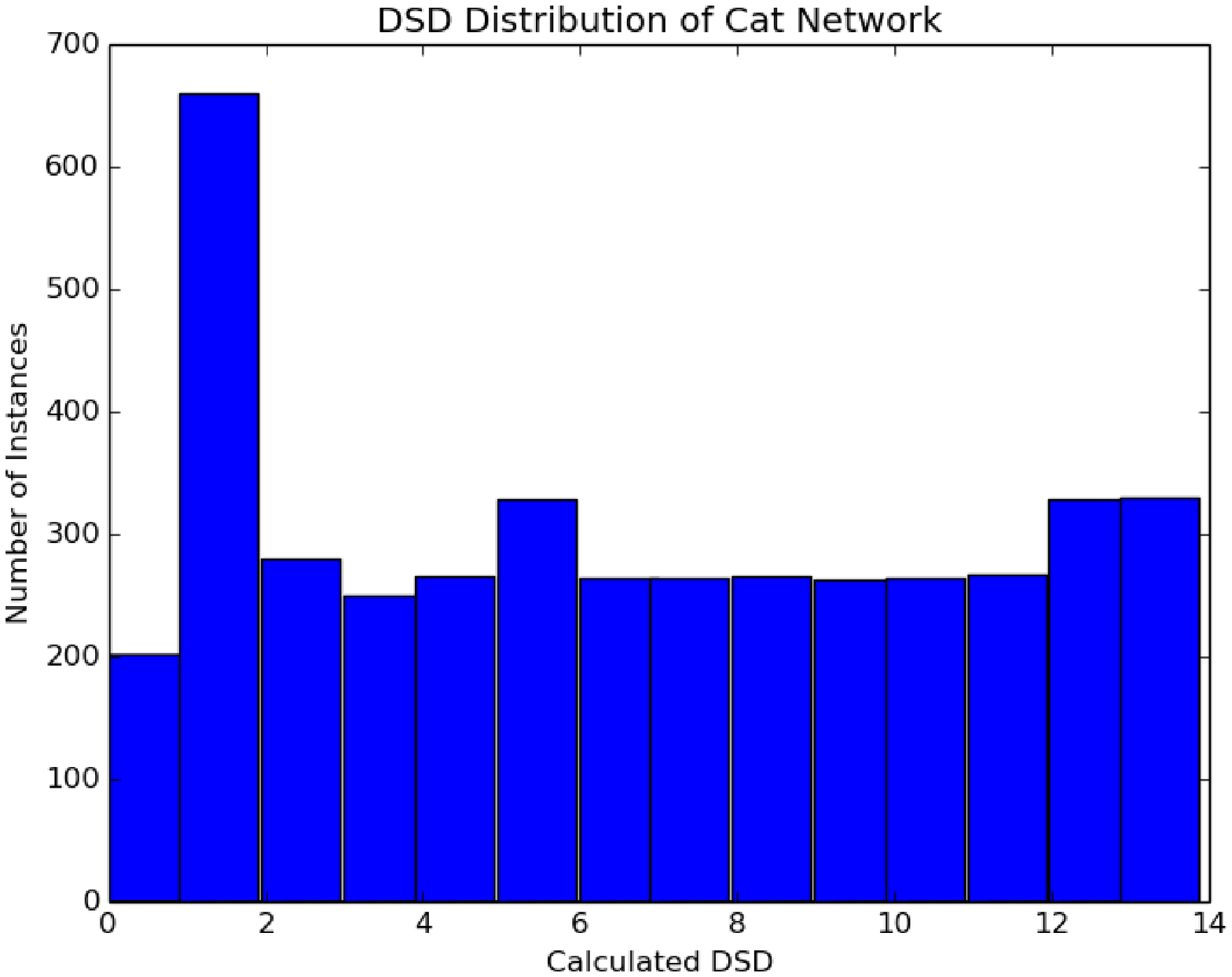, width=0.45\linewidth, height=0.45\linewidth}
\hfil
%\includegraphics[%bb=0 0 585 441,
%width=0.45\linewidth,height=0.45\linewidth]{rhesus_DSD_no_spaces.png}
\psfig{figure=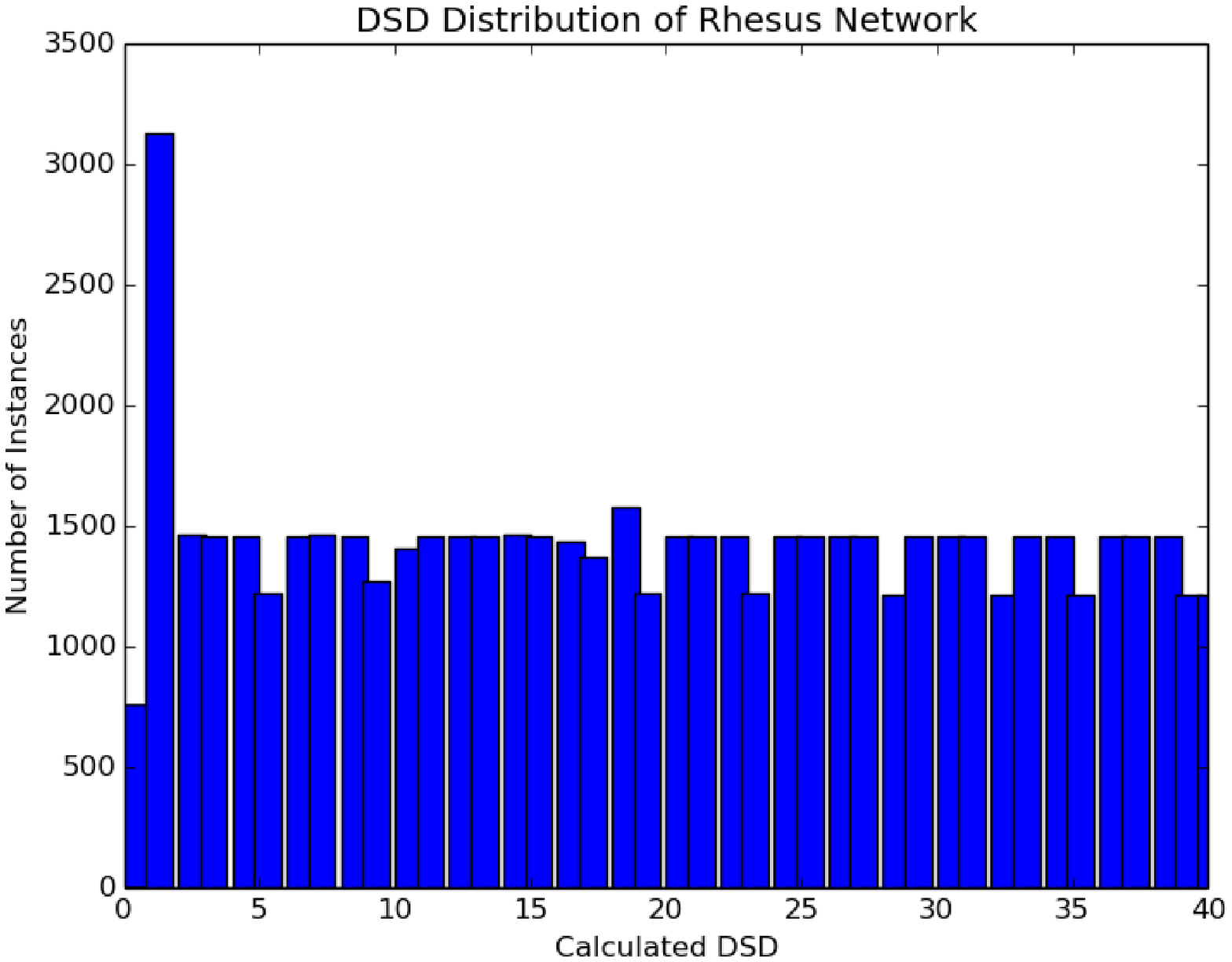, width=0.45\linewidth, height=0.45\linewidth}

% rhesus_cluster.png: 812x612 pixel, 100dpi, 20.62x15.54 cm, bb=0 0 585 441
\caption{The distribution of the DSD $L_1$-distances of brain networks: (a), a Cat; (b): a Rhesus Monkey} 
\label{fig:2}
\end{figure}

Both graphs has a small interval consisting of many
pairs while other values are more or less uniformly distributed.
We think, that phenomenon might be caused by the clustering of
a dense core. The two graphs have many branches sticking out. Since we
are using $L_1$-distance, it doesn't matter the directions of these
branches sticking out when they are embedded into $\mathbb{R}^n$ using
Green's function. 

When we change $L_1$-distance to $L_2$-distance, the pattern should be
broken. This is confirmed in Figure \ref{fig:3}. The actual
distributions are mysterious to us.

\begin{figure}[htbp]
 \centering
%\includegraphics[%bb=0 0 585 441,
%width=0.45\linewidth,height=0.45\linewidth]{cat_DSD_L2_no_space.png}
\psfig{figure=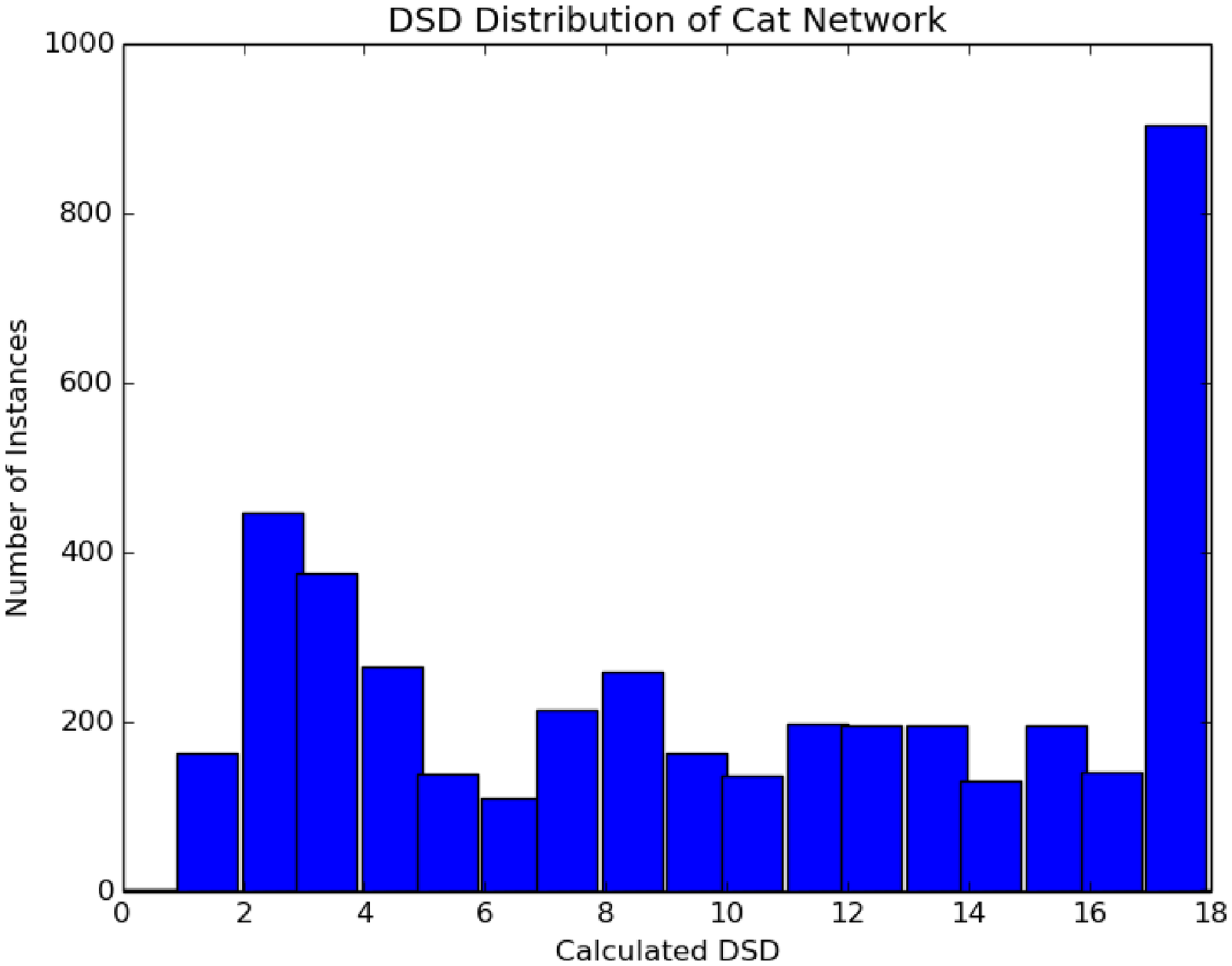, width=0.45\linewidth, height=0.45\linewidth}
\hfil
%\includegraphics[%bb=0 0 585 441,
%width=0.45\linewidth,height=0.45\linewidth]{rhesus_DSD_L2_no_spaces.png}
\psfig{figure=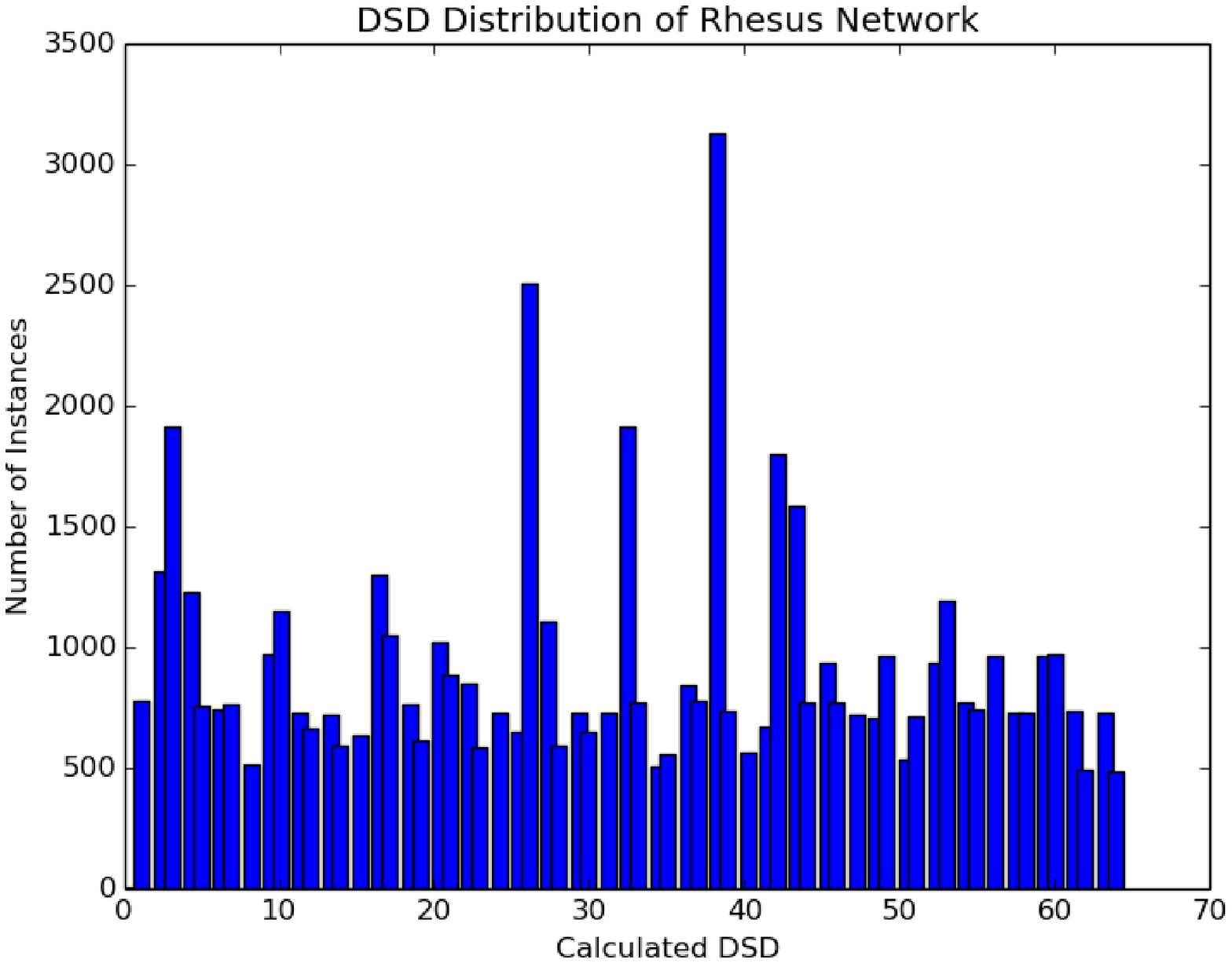, width=0.45\linewidth, height=0.45\linewidth}
% rhesus_cluster.png: 812x612 pixel, 100dpi, 20.62x15.54 cm, bb=0 0 585 441
\caption{The distribution of the DSD $L_2$-distances of brain networks: (a), a Cat; (b): a Rhesus Monkey} 
\label{fig:3}
\end{figure}

 \end{document}